\documentclass{amsart}

\usepackage{a4wide,amscd,amsmath,amssymb,amsthm}
\usepackage[all]{xy}

\theoremstyle{plain}
\newtheorem{thm}{Theorem}[section]

\newtheorem{lem}[thm]{Lemma}

\newtheorem{cor}[thm]{Corollary}

\theoremstyle{definition}

\newtheorem{obs}[thm]{Remark}

\def\Z{\ensuremath{\mathbb{Z}}}
\def\R{\ensuremath{\mathbb{R}}}

\def\ol{\overline}
\def\t{\text}

\numberwithin{equation}{section}

\begin{document}

\title{The Third Homotopy Group as a $\pi_1$--Module}

\date{\today}

\author{Hans-Joachim Baues and Beatrice Bleile}

\address{Max Planck Institut f{\"u}r Mathematik\\
Vivatsgasse 7\\
D--53111 Bonn, Germany}

\email{baues@mpim-bonn.mpg.de}

\address{Second Author's Home Institution: 
School of Science and Technology\\
University of New England\\
NSW 2351, Australia}

\email{bbleile@une.edu.au}

\begin{abstract}
It is well--known how to compute the structure of the second homotopy group of a space, $X$, as a module over the fundamental group, $\pi_1X$, using the homology of the universal cover and the Hurewicz isomorphism. We describe a new method to compute the third homotopy group, $\pi_3 X$, as a module over $\pi_1 X$. Moreover, we determine $\pi_3 X$ as an extension of $\pi_1 X$--modules derived from Whitehead's Certain Exact Sequence. Our method is based on the theory of quadratic modules. Explicit computations are carried out for pseudo--projective $3$--spaces $X = S^1 \cup e^2 \cup e^3$ consisting of exactly one cell in each dimension $\leq 3$. 
\end{abstract}

\subjclass[2000]{}

\maketitle

\section{Introduction}
Given a connected $3$--dimensional CW--complex, $X$, with universal cover, $\widehat X$, Whitehead's Certain Exact Sequence \cite{Wh2} yields the short exact sequence
\begin{equation}\label{ses1}
\xymatrix{
\Gamma \pi_2 X \, \ar@{>->}[r] & \pi_3 X \ar@{->>}[r] & \text{H}_3 \widehat X}
\end{equation}
of $\pi_1$--modules, where $\pi_1 = \pi_ 1(X)$. As a group, the homology $\text{H}_3 \widehat X$ is a subgroup of the free abelian group of cellular $3$--chains of $\widehat X$, and thus itself free abelian. Hence the sequence splits as a sequence of abelian groups. This raises the question whether (\ref{ses1}) splits as a sequence of $\pi_1$--modules -- there are no examples known in the literature.

It is well--known how to compute $\pi_2(X) \cong \text{H}_2 \widehat X$ as a $\pi_1$--module, using the Hurewicz isomorphism, and how to compute $\text{H}_3 \widehat X$ using the cellular chains of the universal cover. In this paper we compute  $\pi_3(X)$ as $\pi_1$--module and (\ref{ses1}) as an extension over $\pi_1$. We answer the question above by providing an infinite family of examples where (\ref{ses1}) does not split over $\pi_1$, as well as an infinite family of examples where it does split over $\pi_1$. As a first surprising example we obtain
\begin{thm}\label{non-trivialpi1action}
There is a connected $3$--dimensional CW--complex $X$ with fundamental group $\pi_1 = \pi_1 X = \Z / 2\Z$, such that $\pi_1$ acts trivially on both $\Gamma \pi_2 X$ and $\text{H}_3 \widehat X$, but non--trivially on $\pi_3 X$. Hence 
\begin{equation*}
\xymatrix{
\Gamma \pi_2 X \, \ar@{>->}[r] & \pi_3 X \ar@{->>}[r] & \text{H}_3 \widehat X}
\end{equation*}
does not split as a sequence of $\pi_1$--modules.
\end{thm}
Below we describe examples for all finite cyclic fundamental groups, $\pi_1$, of even order, where (\ref{ses1}) does not split over $\pi_1$. The examples we consider are CW--complexes,
\[X = S^1 \cup e^2 \cup e^3,\]
with precisely one cell, $e^i$, in every dimension $i = 0, 1, 2, 3$. In general, we obtain such a CW--complex, $X$, by first attaching the 2--cell $e_2$ to $S^1$ via $f \in \pi_1S^1 = \Z$. We assume $f>0$. This yields the 2--skeleton of $X$, $X^2 = P_f$, which is a pseudo--projective plane, see \cite{Olum}. Then $\pi_1 = \pi_1 X = \pi_1 P_f = \Z/f\Z$ is a cyclic group of order $f$. We write $R = \Z [\pi_1]$ for the integral group ring of $\pi_1$ and $K$ for the kernel of the augmentation $\varepsilon: R \rightarrow \Z$. Then the pseudo--projective $3$--space, $X=P_{f,x}$, is determined by the pair, $(f,x)$, of attaching maps, where $x \in \pi_2 P_f = K$ is the attaching map of the 3--cell $e_3$. In this case 
\[\pi_2(X) = \text{H}_2(\widehat X) = K / xR,\]
and 
\[\text{H}_3 \widehat X = \ker (d_x: R \rightarrow R, x \mapsto xy), \]
where $xy$ is the product of $x, y \in R$.

A splitting function $u$ for the exact sequence (\ref{ses1}) is a function between sets, $u: \text{H}_3 \widehat X \rightarrow \pi_3 X$, such that $u(0) = 0$ and the composite of $u$ and the projection $\pi_3 X \twoheadrightarrow \text{H}_3 \widehat X$  is the identity. Such a splitting function determines maps
\[A=A_u: \text{H}_3 \widehat X \times \text{H}_3 \widehat X \rightarrow \Gamma(\pi_2 X)
\quad \text{and} \quad B=B_u: \text{H}_3 \widehat X \rightarrow \Gamma(\pi_2 X),\]
by the cross--effect formul\ae
\[A(y,z) = u(y+z) - (u(y) + u(z)) \quad \text{and} \quad B(y) = (u(y))^1 - u(y^1).\]
Here $B$ is determined by the action of the generator $1$ in the cyclic group $\pi_1$, denoted by $y \mapsto y^1$.
\begin{obs}\label{obsAB}
The functions $A$ and $B$ determine $\pi_3 X$ as a $\pi_1$--module. In fact, the bijection $\text{H}_3 \widehat X \times \Gamma(\pi_2 X) = \pi_3(P_{f,x})$, which assigns to $(y,v)$ the element $u(y)+v$ is an isomorphism of $\pi_1$--modules, where the left hand side is an abelian group by 
\[(y,v) + (z,w) = (y+z, v+w+A(y,z))\]
and a $\pi_1$--module by
\[(y,v)^1 = (y^1, v^1+B(y)).\]
The cross--effect of $B$ satisfies
\begin{eqnarray*}
B(y+z) - (B(y) + B(z)) = (A(y,z))^1 - A(y^1,z^1),
\end{eqnarray*}
such that $B$ is a homomorphism of abelian groups if $A=0$.
\end{obs}
In this paper we describe a method to determine a splitting function $u = u_x$, which, a priori, is not a homomorphism of abelian groups. We investigate the corresponding functions $A$ and $B$ and compute them for a family of examples.
\begin{thm}\label{AandBforex}
Let $X = P_{f,x}$ be a pseudo--projective $3$--space with $x = \tilde x ([\ol{1}] - [\ol{0}]) \in K, \tilde x \in \Z, \tilde x \neq 0$ and $f > 1$. Let $N = \sum_{i=0}^{f-1} [\ol{i}]$ be the norm element in $R$. Then
\[\text{H}_3(\widehat P_{f,x}) = \{ \tilde y N \, | \, \tilde y \in \mathbb Z\} \cong \Z\]
is a $\pi_1$--module with trivial action of $\pi_1$, and
\[\pi_2(P_{f,x}) = (\Z / \tilde x \Z) \otimes_\Z K,\]
with the action of $\pi_1$ induced by the $\pi_1$--module $K$. There is a splitting function $u = u_x$ such that, for $y = \tilde{y} N$ and $z \in \text{H}_3(\widehat P_{f,x})$, the functions $A$ and $B$ are given by
\begin{eqnarray*}
& A(y,z)  & =  \quad 0\\[3pt]
& B(y) & =  \quad - \tilde x \tilde y \gamma q ([\ol{1}] - [\ol{0}]),
\end{eqnarray*}
where $\gamma: \pi_2(P_{f,x}) \rightarrow \Gamma(\pi_2(P_{f,x}))$ is the universal quadratic map for the Whitehead functor $\Gamma$ and $q: K \rightarrow \pi_2(P_{f,x}), k \mapsto 1 \otimes k$. As in \ref{obsAB}, the pair $A, B$ computes $\pi_3 X$ as a $\pi_1$--module. 
\end{thm}
As $\text{H}_3(\widehat X)$ is free abelian, the exact sequence  (\ref{ses1}) always allows a splitting function which is a homomorphism of abelian groups. This leads, for $X = P_{f,x}$, to the injective function
\[\tau: \text{Ext}_{\pi_1}(\text{H}_3(\widehat X), \Gamma(\pi_2 X)) \rightarrowtail \text{coker}(\beta),\]
with
\[\beta: \text{Hom}_\Z(\text{H}_3(\widehat X),  \Gamma(\pi_2 X)) \rightarrow \text{Hom}_\Z(\text{H}_3(\widehat X),  \Gamma(\pi_2X)),t \mapsto \beta_t,\]
given by
\[\beta_t(\ell) = - t(\ell^1) + (t(\ell))^1.\]
The function $\tau$ maps the equivalence class of an extension to the element in $\text{coker}\beta$ represented by $B = B_u$, where $u$ is a $\Z$--homomorphic splitting function for the extension. Hence the equivalence class, $\{\pi_3 X\}$, of the extension $\pi_3 X$ in (\ref{ses1}) is determined by the image $\tau\{\pi_3 X\} \in \text{coker}(\beta)$. For the family of examples in \ref{AandBforex} we show
\begin{thm}\label{extcalc}
Let $X = P_{f,x}$ be a pseudo--projective $3$--space with $x = \tilde x ([\ol{1}] - [\ol{0}]), \tilde x \in \Z , \tilde x \neq 0$ and $f > 1$. Then $\beta: \Gamma((\Z / \tilde x \Z) \otimes_\Z K) \rightarrow \Gamma((\Z / \tilde x \Z) \otimes_\Z K)$ maps $\ell$ to $-\ell + \ell^1$ and $\tau\{\pi_3 X\} \in \text{coker}(\beta)$ is represented by $\tilde x \gamma q ([\ol{1}] - [\ol{0}]) \in \Gamma(\pi_2)$. Hence $\tau\{\pi_3 X\} = 0$ if $\tilde x$ is odd, so that, in this case, $\pi_3 X$ in (\ref{ses1}) is a split extension over $\pi_1$. If both $\tilde x$ and $f$ are even, then $\tau\{\pi_3 X\}$ is a non--trivial element of order $2$, and the extension $\pi_3 X$ in (\ref{ses1}) does not split over $\pi_1$. Moreover, $\tau\{\pi_3 X\}$ is represented by $B$ in \ref{AandBforex}. If $\tilde x$ is even and $f$ is odd, then $\tau\{\pi_3 X\}$ is trivial and the extension $\pi_3 X$ in (\ref{ses1}) does split over $\pi_1$.
\end{thm}
This result is a corollary of \ref{AandBforex}, the computations are contained at the end of Section \ref{examples}.

Given a pseudo--projective $3$--space, $P_{f,x}$, and an element $z \in \pi_3(P_{f,x})$, we obtain a pseudo--projective $4$--space, $X = P_{f,x,z} = S^1 \cup e^2 \cup e^3 \cup e^4$, where $z$ is the attaching map of the $4$--cell $e^4$.
For $n \geq 2$, the attaching map $z$ of an $(n+1)$--cell in a CW--complex, $X$, is \emph{homologically non--trivial} if the image of $z$ under the Hurewicz homomomorphism is non--trivial in $\t{H}_n \widehat{X}^n$. 
\begin{thm}\label{trivialonpi_3}
Let $X = S^1 \cup e^2 \cup e^3 \cup e^4$ be a pseudo--projective $4$--space with $\pi_1X = \Z / 2 \Z$ and homologically non--trivial attaching maps of cells in dimension $3$ and $4$. Then the action of $\pi_1 X$ on $\pi_3 X$ is trivial.
\end{thm}
Theorem \ref{trivialonpi_3} is a corollary to Theorem \ref{trivialonpi_3detail}.

\section{Crossed Modules}
We recall the notions of pre-crossed module, Peiffer commutator, crossed module and nil(2)--module, which are ingredients of algebraic models of $2$-- and $3$--dimensional CW--complexes used in the proofs of our results, see \cite{Baues} and \cite{Brown}. In particular, Theorem \ref{1.8} provides an exact sequence in the algebraic context of a nil(2)--module equivalent to Whitehead's Certain Exact Sequence (\ref{ses1}). 

A \emph{pre--crossed module} is a homomorphism of groups, $\partial: M \rightarrow N$, together with an action of $N$ on $M$, such that, for $x \in M$ and $\alpha \in N$,
\[\partial(x^\alpha) = - \alpha + \partial x + \alpha.\]
Here the action is given by $(\alpha, x) \mapsto x^\alpha$ and we use additive notation for group operations even where the group fails to be abelian. The \emph{Peiffer commutator} of $x, y \in M$ in such a pre--crossed module is given by
\[\langle x, y \rangle = - x - y + x + y^{\partial x}.\]
The subgroup of $M$ generated by all iterated Peiffer commutators $\langle x_1, \ldots, x_n\rangle$ of length $n$ is denoted by $P_n(\partial)$ and a \emph{nil(n)--module} is a pre--crossed module $\partial: M \rightarrow N$ with $P_{n+1} (\partial)= 0$. A \emph{crossed module} is a nil(1)--module, that is, a pre--crossed module in which all Peiffer commutators vanish. We also consider nil(2)--modules, that is, pre--crossed modules for which $P_3(\partial) = 0$.

A morphism or map $(m,n): \partial \rightarrow \partial'$ in the category of pre--crossed modules is given by a commutative diagram
\[\xymatrix{
M \ar[r]^m \ar[d]_\partial & M' \ar[d]^{\partial'} \\
N \ar[r]_n& N'}\]
in the category of groups, where $m$ is $n$--equivariant, that is, $m(x^\alpha) = m(x)^{n(\alpha)}$, for $x \in M$ and $\alpha \in N$. The categories of crossed modules and nil(2)--modules are full subcategories of the category of pre--crossed modules. 

Note that $P_{n+1} (\partial) \subseteq \ker \partial$ for any pre--crossed module, $\partial: M \rightarrow N$. Thus we obtain the \emph{associated nil(n)--module} $r_n(\partial): M/P_{n+1}(\partial) \rightarrow N$, where the action on the quotient is determined by demanding that the quotient map $q: M \rightarrow M/P_{n+1}(\partial)$ be equivariant. For $n=1$ we write $\partial^{cr} = r_1(\partial): M^{cr} = M/P_{2}(\partial) \rightarrow N$ for the \emph{crossed module associated} to $\partial$.

Given a set, $Z$, let $\langle Z \rangle$ denote the free group generated by $Z$. Now take a group, $N$, and a group homomorphism, $f: F = \langle Z \rangle \rightarrow N$. Then the \emph{free $N$--group} generated by $Z$ is the free group, $\langle Z \times N \rangle$, generated by elements denoted by $x^\alpha = ((x,\alpha))$ with $x \in Z$ and $\alpha \in N$. These are elements in the product $Z\times N$ of sets. The action is determined by
\begin{equation}\label{freegrpgen}
((x,\alpha))^\beta = ((x, \alpha + \beta)).
\end{equation}
Define the group homomorphism $\partial_f: \langle Z \times N \rangle \rightarrow N$ by $((x, \alpha)) \mapsto -\alpha + f(x) + \alpha$, for generators $((x,\alpha)) \in Z \times N$, to obtain the pre--crossed module $\partial_f$ with associated nil(n)--module $r_n(\partial_f): \langle Z \times N \rangle/P_{n+1} (\partial_f) \rightarrow N$. Note that $r_n(\partial_f) \iota = f$, where $\iota = p \iota_F$ is the composition of the inclusion $\iota_F: F= \langle Z \rangle \rightarrow \langle Z  \times N \rangle$ and the projection $p: \langle Z \times N \rangle \rightarrow M = \langle Z \times N \rangle/P_{n+1}(\partial_f)$ onto the quotient. 

\begin{obs} The nil($n$)--module, $r_n(\partial_f): M=\langle Z \times N \rangle/P_{n+1} (\partial_f) \rightarrow N$, satisfies the following universal property: For every nil($n$)--module, $\partial': M' \rightarrow N'$, and every pair of group homomorphisms, $m_F: F = \langle Z \rangle \rightarrow M'$, and $n: N \rightarrow N'$ with $\partial' m_F = n f$, there is a unique group homomorphism, $m: M \rightarrow M'$, such that $m \iota = m_F$, and $(n,m): r_n(\partial_f) \rightarrow \partial'$ is a map of nil(n)--modules.
\[\xymatrix{
M \ar@{-->}[rr]^m \ar[dd]_{r_n(\partial_f)} && M' \ar[dd]^{\partial'} \\
& F \ar[ul]_{\iota} \ar[ur]_{m_F} \ar[dl]^f & \\
N \ar[rr]_n && N'
}\]
Thus $r_n(\partial_f)$ is called the \emph{free nil(n)--module with basis} $f$. A free nil($n$)--module is \emph{totally free} if $N$ is a free group.
\end{obs}

Given a path connected space $Y$ and a space $X$ obtained from $Y$ by attaching $2$--cells, let $Z_2$ be the set of $2$--cells in $X - Y$, and let $f: Z_2 \rightarrow \pi_1(Y)$ be the attaching map. J.H.C. Whitehead \cite{Wh1} showed that 
\begin{equation}\label{whitehead}
\partial: \pi_2(X,Y) \rightarrow \pi_1(Y)
\end{equation}
is a free crossed module with basis $f$.  Then $\ker \partial = \pi_2(X), \rm{coker}\, \partial = \pi_1(X)$ and $\partial$ is totally free if $Y$ is a one--point union of $1$--spheres. Whitehead also proved that the abelianisation of the group $\pi_2(X,Y)$ is the free $R$--module ${\langle Z_2 \rangle}_R$ generated by the set $Z_2$, where $R = \Z[\pi_1(X)]$ is the group ring \cite{Wh1}. 

Now take a totally free nil(2)--module $\partial: M \rightarrow N$ with associated crossed module $\partial^{cr}: M^{cr} \rightarrow N$. Let
\[\xymatrix{
M \ar@{->>}[r]^-q & M^{cr} \ar@{->>}[r]^-{h_2} & C = (M^{cr})^{ab}}\]
be the composition of projections. Put $K = h_2(\ker(\partial^{cr}))$. Further, let $\Gamma$ be \emph{Whitehead's quadratic functor} and $\tau: \Gamma(K) \rightarrowtail K \otimes K \subset C \otimes C$ the composition of the injective homomorphism induced by the quadratic map $K \rightarrow K \otimes K, k \mapsto k \otimes k$ and the inclusion. The \emph{Peiffer commutator map}, $w: C \otimes C \rightarrow M$, is given by $w(\{x\} \otimes \{y\}) = \langle x, y \rangle$, for $x, y \in M$ with $\{x\} = h_2(q(x)), \{y\} = h_2(q(y))$. Lemma (IV 1.6) and Theorem (IV 1.8) in \cite{Baues} imply
\begin{thm}\label{1.8}
Let $\partial: M \rightarrow N$ be a totally free nil(2)--module. Then the sequence
\[\xymatrix{
\Gamma(K)\,\,  \ar@{>->}[r]^-\tau & C \otimes C \ar[r]^-w & M \ar@{->>}[r]^-q & M^{cr}}\]
is exact and the image of $w$ is central in M.
\end{thm}

\section{Pseudo--Projective Spaces in Dimensions 2 and 3}\label{P_fandP_{f,x}}
Real projective $n$--space $\R\t{P}^n$ has a cell structure with precisely one cell in each dimension $\leq n$. More generally, a CW--complex,
\[X = S^1 \cup e^2 \cup \ldots \cup e^n,\]
with precisely one cell in each dimension $\leq n$, is called a \emph{pseudo--projective $n$--space}. For $n = 2$ we obtain \emph{pseudo--projective planes}, see \cite{Olum}. In this section we fix notation and consider pseudo--projective spaces in dimensions $2$ and $3$. In particular, we determine the totally free crossed module associated with a pseudo--projective plane and begin to investigate the totally free nil(2)--module associated with a pseudo--projective $3$--space.

The fundamental group of a pseudo--projective plane $P_f = S^1 \cup e^2$, with attaching map $f \in \pi_1(S^1) = \Z$, is the cyclic group $\pi_1= \pi_1(P_f) = \Z / f \Z$. We obtain $\pi_1 = \Z$ for $f = 0$, $\pi_1 = \{ 0\}$ for $f = 1$, and the bijection of sets
\[ \{ 0, 1, 2, \ldots, f-1 \} \rightarrow \pi_1 = \Z/f\Z,\quad k \mapsto \overline{k} = k + f\Z,\]
for $1 < f$. Addition in $\pi_1$ is given by
\[ \ol{k} + \ol{\ell} = 
\begin{cases}
 \ol{k+\ell} \quad & \text{for} \quad k+\ell < f;\\
\ol{k+\ell-f} \quad & \text{for} \quad k+\ell \geq f. \end{cases}\]
Denoting the integral group ring of the cyclic group $\pi_1$ by $R = \Z[\pi_1]$, an element $x \in R$ is a linear combination
\[ x = \sum_{\alpha \in \pi_1} x_{\alpha} [\alpha] =  \sum_{k = 0}^{f-1} x_{\ol{k}} [\ol{k}],\]
with $x_{\alpha}, x_{\ol{k}} \in \Z$. Note that $1_R=[\ol{0}]$ is the neutral element with respect to multiplication in $R$ and, for $x = \sum_{\alpha \in \pi_1} x_{\alpha} [\alpha] , y = \sum_{\beta \in \pi_1} y_{\beta} [\beta] $,
\[ xy = \sum_{\alpha, \beta \in \pi_1} \, x_{\alpha} y_{\beta} \, [\alpha+\beta] = \sum_{\ell=0}^{f-1} \big(\sum_{k = 0}^\ell x_{\ol{k}} \,y_{\ol{\ell-k}} + \sum_{k=\ell+1}^{f-1} x_{\ol{k}} \,y_{\ol{f+\ell-k}} \big) [\ol{\ell}].\]
The \emph{augmentation} $\varepsilon = \varepsilon_R: R \rightarrow \Z$ maps $\sum_{\alpha \in \pi_1} x_{\alpha} [\alpha]$ to $\sum_{\alpha \in \pi_1} x_{\alpha}$. The \emph{augmentation ideal}, $K$, is the kernel of $\varepsilon$. For a right $R$--module, $C$, we write the action of $\alpha \in \pi_1$ on $x \in C$ exponentially as $x^{\alpha} = x [\alpha]$. 

Given a pseudo--projective plane $P_f = S^1 \cup e^2$ with attaching map $f \in \pi_1(S^1) = \Z$, Whitehead's results on the free crossed module (\ref{whitehead}) imply that
\begin{equation}\label{P_ftotallyfreecross}
\partial: \pi_2(P_f, S^1) \rightarrow \pi_1(S^1)
\end{equation}
is a totally free crossed module with one generator, $e_i$, in dimensions $i = 1,2$, and basis $\tilde{f}: Z_2 = \{e_2\} \rightarrow \pi_1(S^1)$ given by $\tilde{f}(e_2) = f e_1$. Note that $\partial$ has cokernel $\pi_1(P_f) = \Z / f \Z = \pi_1$ and kernel $\pi_2(P_f)$.
\begin{lem}\label{ccmiso}
The diagram
\[\xymatrix{
\pi_2(P_f, S^1) \ar[r]^\partial \ar[d]_\cong & \pi_1(S^1) \ar@{=}[d] \\
R \ar[r]_{f \cdot \varepsilon_R} & \Z }\]
is an isomorphism of crossed modules, where $\varepsilon_R: R \rightarrow \mathbb Z$ is the augmentation.
\end{lem}
\begin{proof}
By Whitehead's results \cite{Wh1} on the free crossed module (\ref{whitehead}), it is enough to show that $\pi_2(P_f, S^1)$ is abelian. As $\partial$ is a totally free crossed module with basis $\tilde{f}$, $\pi_2(P_f, S^1)$ is generated by elements $e^n = ((e_2, n))$, see (\ref{freegrpgen}). Note that we obtain $e^n$ by the action of $n \in \Z$ on $\iota(e_2) = ((e_2,0)) = e^0$ and $\partial(e^n) = -n + \partial e + n = \partial e = f$ as $\pi_1(S^1) = \Z$ is abelian. We obtain
\begin{eqnarray*}
\langle e^n, e^m \rangle - \langle e^m, e^m \rangle
& = & -e^n - e^m + e^n + (e^m)^{\partial(e^n)} - (- e^m - e^m + e^m + (e^m)^{\partial(e^m)}) \\
& = & -e^n - e^m + e^n + (e^m)^{f} - (e^m)^{f} + e^m \\
& = & (e^n,e^m),
\end{eqnarray*}
where $(a,b) = -a - b + a + b$ denotes the commutator of $a$ and $b$. Thus commutators of generators are sums of Peiffer commutators which are trivial in a crossed module.
\end{proof}

With the notation of Theorem \ref{1.8} and $M = \pi_2(P_f, S^1)$, Lemma \ref{ccmiso} shows that $M = M^{cr} = (M^{cr})^{ab} = R$ and that $\pi_2(P_f) = \ker \partial = \ker \partial^{cr} = \ker (f \cdot \varepsilon) = K$ is the augmentation ideal of $R$, for $f \neq 0$. Thus the homotopy type of a pseudo--projective 3--space,
\begin{equation}\label{P_{f,x}}
P_{f,x} = S^1 \cup e^2 \cup e^3,
\end{equation}
is determined by the pair $(f,x)$ of attaching maps, $f \in \pi_1(S^1) = \Z$ of the 2-cell $e^2$, and $x \in \pi_2(P_f) = K \subseteq R$ of the 3--cell $e^3$. We obtain the totally free nil(2)--module 
\begin{equation}\label{tfnil2}
\xymatrix{
M = \pi_2(P_{f,x}, S^1) \ar[r]^-\partial  & N = \pi_1(S^1).}
\end{equation}
In the next section we use Theorem \ref{1.8} to describe the group structure of $\pi_2(P_{f,x}, S^1)$, as well as the action of $N$ on $\pi_2(P_{f,x}, S^1)$. The formul{\ae} we derive are required to compute the homotopy group $\pi_3(P_{f,x})$ as a $\pi_1$--module.

\section{Computations in nil(2)--Modules}\label{nil2comp}
In this Section we consider totally free nil(2)--modules, $\partial: M \rightarrow N$, generated by one element, $e_i$, in dimensions $i=1, 2$, with basis $\tilde f: \{e_2\} \rightarrow N \cong \Z$. Then $\pi_1 = \text{coker} \partial = \Z / f \Z$ and, with $R = \Z [\pi_1]$, we obtain $(M^{cr})^{ab} =  C = R$. Thus Theorem \ref{1.8} yields the short exact sequence
\begin{equation}\label{ses3}
\xymatrix{
(R \otimes R) / \Gamma(K) \, \, \ar@{>->}[r]^-w & M \ar@{->>}[r]^q & R}
\end{equation}
with the image of $(R \otimes R) / \Gamma(K)$ central in $M$. This allows us to compute the group structure of $M$, as well as the action of $N = \Z$ on $M$, by computing the cross--effects of a set--theoretic splitting $s$ of (\ref{ses3}) with respect to addition and the action of $N$, even though here $M$ need not be commutative. 

The element $x \otimes y \in R \otimes R$ represents an equivalence class in $R \otimes R / \Gamma(K)$, also denoted by $x \otimes y$, so that $w(x\otimes y) = \langle \hat{x}, \hat{y}\rangle$ is the Peiffer commutator for $x, y \in R$, with $x = q(\hat{x})$ and $y = q(\hat{y})$. As a group, $M$ is generated by elements $e^n = ((e_2, n))$, in particular, $e = e^0 = ((e_2, 0))$, see (\ref{freegrpgen}). We write
\[ k e^n = 
\begin{cases} e^n + \ldots + e^n  \quad (k \, \text{summands}) & \quad \text{for}\, k>0, \\
0 & \quad \text{for}\, k=0 \, \text{and} \\
-e^n - \ldots - e^n   \quad (-k \, \text{summands}) & \quad \text{for}\, k<0, \end{cases}\]
and define the set-theoretic splitting $s$ of (\ref{ses3}) by
\[s: R \longrightarrow M, \quad \sum_{k = 0}^{f-1} x_{\ol{k}} [\ol{k}] \longmapsto x_{\ol{0}} e^0 + x_{\ol{1}} e^1 + \ldots + x_{\ol{f-1}} e^{f-1}.\]
Then every $m \in M$ can be expressed uniquely as a sum $m = s(x) + w(m^\otimes)$ with $x \in R$ and $m^\otimes \in (R \otimes R) / \Gamma(K)$. The following formul{\ae} for the cross--effects of $s$ with respect to addition and the action provide a complete description of the nil(2)--module $M$ in terms of $R$ and $R \otimes R/ \Gamma(K)$. 

Given a function, $f: G \rightarrow H$, between groups, $G$ and $H$, we write
\begin{equation}\label{f(x|y)}
f(x|y) = f(x+y) - (f(x) + f(y)), \quad \text{for $x, y \in G$}.
\end{equation}
\begin{lem}\label{scross1}
Take $x= \sum_{m=0}^{f-1} x_{\ol{m}}, y=\sum_{n=0}^{f-1} y_{\ol{n}} \,[\ol{n}] \in R$. Then 
\[s(x|y) =  w (\nabla(x,y)),\]
where
\[\nabla(x,y) = \sum_{m=1}^{f-1} \sum_{n=0}^{m-1} x_{\ol{m}} \, y_{\ol{n}} w( [\ol{n}] \otimes [\ol{m}] - [\ol{m}] \otimes [\ol{m}] ).\]
Thus $\nabla(x,y)$ is linear in $x$ and $y$, yielding a homomorphism $\nabla: R \otimes R \rightarrow R \otimes R$.
\end{lem}
\begin{proof}
First note that, by definition, $\nabla(k[\ol{m}], \ell[\ol{n}]) = 0$ unless $m>n$. To deal with the latter case, recall that commutators are central in $M$ and use induction, first on $k$, then on $\ell$, to show that
\begin{eqnarray*}
(k e^m, \ell e^n) = k\ell (e^m, e^n),
\end{eqnarray*}
for $k, \ell > 0$. To show equality for negative $k$ or $\ell$, replace $e^m$ or $e^n$ by $-e^m$ and $-e^n$, respectively. Furthermore, note that the equality
\begin{equation}
(e^n, e^m) =  -e^n - e^m + e^n + e^m = \langle e^n, e^m \rangle - \langle e^m, e^m \rangle
\end{equation}
for commutators of generators of totally free cyclic crossed modules derived in the proof of Lemma \ref{ccmiso} holds in any totally free nil(n)--module generated by one element in each dimension. Taking $x = \sum_{m=0}^{f-1} x_{\ol{m}} \, [\ol{m}]$ and $y = \sum_{n=0}^{f-1} y_{\ol{n}} \,[\ol{n}]$, we obtain
\begin{eqnarray*}
&& s(x+y) \\
&& \quad = (x_{\ol{0}} + y_{\ol{0}}) \, e + \ldots + (x_{\ol{m}} + y_{\ol{m}}) \, e^m + \ldots + (x_{\ol{f-1}}+ y_{\ol{f-1}}) \, e^{f-1} \\
&& \quad = (x_{\ol{i}} \, e + \ldots + x_{\ol{f-1}} \, e^{f-1}) +  ( y_{\ol0} \, e + \ldots + y_{\ol{f-1}} \, e^{f-1}) + \sum_{m=1}^{f-1} \sum_{n=0}^{m-1} x_{\ol{m}} \, y_{\ol{n}} \, (e^{n}, e^m)\\
&& \quad = s(x) + s(y) + \sum_{m=1}^{f-1} \sum_{n=0}^{m-1} x_{\ol{m}} \, y_{\ol{n}} \, \big( \langle e^n, e^m \rangle - \langle e^m, e^m \rangle \big)\\
&& \quad = s(x) + s(y) + \sum_{m=1}^{f-1} \sum_{n=0}^{m-1} x_{\ol{m}} \, y_{\ol{n}} w( [\ol{n}] \otimes [\ol{m}] - [\ol{m}] \otimes [\ol{m}] ).
\end{eqnarray*}
\end{proof}
\begin{cor}\label{scrosscor}
Take $x \in R$ and $r \in \Z$. Then 
\[ s(rx) = rs(x) + \binom{r}{2} w(\nabla(x,x)), \quad \text{where} \quad \binom{r}{2} = \frac{r(r-1)}{2}. \]
\end{cor}
As $N = \Z$ is cyclic, the action of $N$ on $M$ is determined by the action of the generator, $1 \in \Z$. The  formula for general $k \in Z$ provided in the next lemma is required for the definition of the set--theoretic splitting $u_x$ of (\ref{ses1}) and the explicit computation of $A$ and $B$ in Theorem \ref{AandBforex}.
\begin{lem}\label{scross2}
Take $x = \sum_{n=0}^{f-1} x_{\ol{n}} [\ol{n}] \in R$ and $\ol{k} \in \pi_1$. Write $R = \Z[\ol{0}, \ldots, \ol{f-1}] = R_k \times \widehat{R}_k$, where $R_k = \Z[\ol{0}, \ldots, \ol{f-k-1}] $ and $\widehat{R}_k = \Z[\ol{f-k}, \ldots, \ol{f-1}] $. Then 
\[\big(s(x)\big)^k = s(x^{\ol{k}}) + w(\ol{\nabla}_k(a,b)),\] 
where $x = (a,b)$ and
\[\ol{\nabla}_k: R_k \times \widehat{R}_k \rightarrow R \otimes R, \quad
(a,b) \mapsto Q_k(a,b) + L_k (b)\]
with
\begin{eqnarray*}
Q_k(a,b) & = &  \sum_{p=0}^{f-\ell-1} \sum_{q=0}^{\ell-1} \, x_{\ol{p}} \, x_{\ol{q+f-\ell}}\, ( [\ol{p + \ell}] \otimes [\ol{q}] - [\ol{q}] \otimes [\ol{q}])   \\
L_k(b) & = & \sum_{q=0}^{\ell-1} \, x_{\ol{q+f-\ell}} \, [\ol{q}] \otimes [\ol{q}].
\end{eqnarray*}
Thus $Q_k$ is linear in $a$ and $b$ and $L_k$ is linear in $b$.
\end{lem}
\begin{proof}
For $\ol{j} \in \pi_1$ and $p \in \mathbb Z$,
\begin{eqnarray*}
e^{j+f} 
& = & (e^j)^{\partial(e)} \\
& = & e^j + (e^j, e) + \langle e, e^j \rangle \\
& = & e^j - ( \langle e, e^j \rangle - \langle e^j, e^j \rangle) +  \langle e, e^j \rangle \\
& = & e^j + \langle e^j, e^j \rangle. \\
\end{eqnarray*}
Thus, for $\ol{n}, \ol{k} \in \pi_1$, with $\ol{n} + \ol{k} = \ol{j}$,
\begin{eqnarray*}
\big(s([\ol{n}])\big)^k
& = & \begin{cases} 
e^j, & \text{for $0 \leq n < f-k$}, \\
e^j + \langle e^j, e^j \rangle, & \text{for $f-k \leq n <f$}
\end{cases}\\
& = & \begin{cases} 
s([\ol{n}]^{\ol{k}}), & \text{for $0 \leq n < f-k$}, \\
s([\ol{n}]^{\ol{k}}) + w([\ol{j}] \otimes [\ol{j}]), & \text{for $f-k \leq n <f$.}
\end{cases}
\end{eqnarray*}
Hence, for $x = \sum_{p=0}^{f-1} x_{\ol{p}} \, [\ol{p}]$,
\begin{eqnarray*}
&&\big(s(x)\big)^k \\
&& =  x_{\ol{0}} \, s([\ol{0}])^k + x_{\ol{1}} \, s([\ol{1}])^k + \ldots + x_{\ol{f-1}} \, s([\ol{f-1}])^k \\
&& =  x_{\ol{0}} \, s([\ol{0}]^{\ol{k}}) +  x_{\ol{1}} \, s([\ol{1}]^{\ol{k}}) + \ldots + x_{\ol{f-1}} \, s([\ol{f-1}]^{\ol{k}})  + \sum_{n=f-k}^{f-1} x_{\ol{n}} \, w([\ol{n+k-f}] \otimes [\ol{n+k-f}]) \\
&& = x_{\ol{f-k}} \, s([\ol{f-k}]^{\ol{k}} ) + \ldots + x_{\ol{f-1}} \, s([\ol{f-1}]^{\ol{k}}) + x_{\ol{0}} \, s([\ol{0}]^{\ol{k}}) + \ldots +  x_{\ol{f-k-1}} \, s([\ol{f-k-1}]^{\ol{k}}) \\
&& \quad + \sum_{p=0}^{f-k-1} \sum_{n=f-k}^{f-1} \, (x_{\ol{p}} s([\ol{p} + \ol{k}]), x_{\ol{n}} s([\ol{n} + \ol{k}])) + \sum_{q=0}^{k-1} x_{\ol{q+f-k}} \, w([\ol{q}] \otimes [\ol{q}]) \\
&& = s(x^{\ol{k}}) + \sum_{p=0}^{f-k-1} \sum_{q=0}^{k-1} \, x_{\ol{p}} \, x_{\ol{q+f-k}}\, w( [\ol{p + k}] \otimes [\ol{q}] - [\ol{q}] \otimes [\ol{q}]) + \sum_{q=0}^{k-1}x_{\ol{q+f-k}} \, w([\ol{q}] \otimes [\ol{q}]).
\end{eqnarray*}
\end{proof}
\begin{obs}
We use the final results of this section to define and establish the properties of the set--theoretic splitting $u_x$ of (\ref{ses1}). The next result shows how the cross--effects interact with multiplication in $R$.
\end{obs}
\begin{lem}\label{mu}
Take $x, y \in R$. Then
\[\sum_{i=0}^{f-1} y_{\ol{i}} \big(s(x)\big)^i = s(xy) + w(\mu(x,y)),\]
where $\mu: R \times R \rightarrow R \otimes R$ is given by
\[\mu(x,y) = - \sum_{i < j} y_{\ol{i}} \, y_{\ol{j}} \,\nabla(x^{\ol{i}}, x^{\ol{j}}) + \sum_{i=0}^{f-1} \big( \ol{\nabla}_i(y_{\ol{i}} x) - \binom{y_{\ol{i}}}{2} \nabla(x,x)^{\ol{i}} \big).\]
\end{lem}
\begin{proof}
By Lemmata \ref{scross1} and \ref{scross2} and Corollary \ref{scrosscor}, we obtain, for $x, y \in R$,
\begin{eqnarray*}
\sum_{i=0}^{f-1} y_{\ol{i}} \big(s(x)\big)^i 
& = & \sum_{i=0}^{f-1}  \big(y_{\ol{i}} s(x)\big)^i \\
& = & \sum_{i=0}^{f-1}  \big( s(y_{\ol{i}} x) - \binom{y_{\ol{i}}}{2} w(\nabla(x,x))\big)^i\\
& = & \sum_{i=0}^{f-1}  s(y_{\ol{i}} x^{\ol{i}}) +  w(\ol{\nabla}_i( y_{\ol{i}} x))  - \big( \binom{y_{\ol{i}}}{2} w(\nabla(x,x))\big)^i \\
& = & s(\sum_{i=0}^{f-1} y_{\ol{i}} \, x^{\ol{i}}) - \sum_{i < j} w(\nabla(y_{\ol{i}} \, x^{\ol{i}}, y_{\ol{j}} \, x^{\ol{j}}))  + \sum_{i=0}^{f-1}  w(\ol{\nabla}_i(y_{\ol{i}} x)) - \binom{y_{\ol{i}}}{2} w(\nabla(x,x)^{\ol{i}}).\\
\end{eqnarray*}
\end{proof}

Finally, the definitions and a simple calculation yield
\begin{lem}\label{mu(x,y|z)}
For $x, y, z \in R$ and with the notation in (\ref{f(x|y)}), 
\[\mu(x,y|z) = - \sum_{i < j} (y_{\ol{i}} \, z_{\ol{j}} + z_{\ol{i}} \, y_{\ol{j}}) \,\nabla(x^{\ol{i}}, x^{\ol{j}}) + 2 \sum_{i=1}^{f-1}y_{\ol{i}} z_{\ol{i}} Q_i(x) - \sum_{i=0}^{f-1} y_{\ol{i}} z_{\ol{i}} \nabla(x,x)^{\ol{i}}.\]
Hence, for fixed $x \in R, \mu(x, \,\,): R\times R \rightarrow R \otimes R, (y,z) \mapsto \mu(x,y|z)$ is bilinear.
\end{lem}

\section{Quadratic Modules}
In dimension $3$, quadratic modules assume the role played by crossed modules in dimension $2$. We recall the notion of quadratic modules and totally free quadratic modules, see \cite{Baues}, which we require for the description of the third homotopy group $\pi_3(P_{f,x})$ of a $3$--dimensional pseudo--projective space $P_{f,x}$, as in (\ref{P_{f,x}}). 

A \emph{quadratic module} $(\omega, \delta, \partial)$ consists of a commutative diagram of group homomorphisms
\[\xymatrix{
& C \otimes C \ar[dl]_\omega \ar[d]^w & \\
L \ar[r]^\delta & M \ar[r]^\partial & N,}\]
such that
\begin{itemize}
\item $\partial: M \rightarrow N$ is a nil(2)--module with quotient map $M \twoheadrightarrow C = (M^{cr})^{ab}, x \mapsto \{ x\}$, and Peiffer commutator map $w$ given by $w(\{x\} \otimes \{y\}) = \langle x, y \rangle$;
\item the \emph{boundary homomorphisms} $\partial$ and $\delta$ satisfy $\partial \delta = 0$, and the \emph{quadratic map} $\omega$ is a lift of $w$, that is, for $x, y \in M$,
\[\delta \omega (\{ x \} \otimes \{ y \} ) = \langle x, y \rangle;\]
\item $N$ acts on $L$, all homomorphisms are equivariant with respect to the action of $N$ and, for $a \in L$ and $x \in M$,
\begin{equation}\label{qdef1}
a^{\partial(x)} = a + \omega(\{\delta a\} \otimes \{x\} + \{x\} \otimes \{\delta a\});
\end{equation}
\item finally, for $a, b \in L$, 
\begin{equation}\label{qdef2}
(a,b) = -a-b+a+b = \omega(\{\delta a\} \otimes \{ \delta b\}).
\end{equation}
\end{itemize}
A \emph{map} $\varphi: (\omega, \delta, \partial) \rightarrow (\omega', \delta', \partial')$ of quadratic modules is given by a commutative diagram
\[\xymatrix{
C \otimes C \ar[d]_{\varphi_{\ast} \otimes \varphi_{\ast}} \ar[r]^\omega & L \ar[d]^l \ar[r]^\delta & M \ar[d]^m\ar[r]^\partial & N \ar[d]^n\\
C' \otimes C' \ar[r]^{\omega'} & L' \ar[r]^{\delta'} & M' \ar[r]^{\partial'} & N' }\]
where $l$ is $n$--equivariant, and $(m,n)$ is a map between pre--crossed modules inducing $\varphi_\ast: C \rightarrow C'$. 

Given a nil(2)--module $\partial: M \rightarrow N$, a free group $F$ and a homomorphism $\tilde f: F \rightarrow M$ with $\partial \tilde f = 0$, a quadratic module $(\omega, \delta, \partial)$ is \emph{free with basis} $\tilde f$, if there is a homomorphism $i: F \rightarrow L$ with $\delta i = \tilde f$, such that the following universal property is satisfied: For every quadratic module $(\omega', \delta', \partial')$ and map $(m,n): \partial \rightarrow \partial'$ of nil(2)--modules and every homomorphism $l_F: F \rightarrow L'$ with $m\tilde f = \delta' l_F$, there is a unique map $(l,m,n)$ of quadratic modules with $li=l_F$.
\[\xymatrix{
L \ar[rr]^\delta \ar@{-->}[dd]_l && M \ar[rr]^\partial \ar[dd]^m && N \ar[dd]^n \\
& F \ar[ul]_i \ar[ur]_{\tilde f} \ar[dl]^{l_F} &&& \\
L' \ar[rr]_{\delta'} && M' \ar[rr]_{\partial'} && N'}\]
For $F = \langle Z \rangle$, the homomorphism $\tilde f$ is determined by its restriction $\tilde f|_Z$ which is then called a \emph{basis} for $(\omega, \delta, \partial)$. A  quadratic module $(\omega, \delta, \partial)$ is \emph{totally free} if it is free, if $\partial$ is a free nil(2)--module and if $N$ is a free group.

\section{The Homotopy Group $\pi_3$ of a Pseudo--Projective 3--Space and the Associated Splitting Function $u_x$}\label{pi_3{P_{f,x}}}
In this section we return to pseudo--projective 3--spaces
\begin{equation*}
P_{f,x} = S^1 \cup e^2 \cup e^3,
\end{equation*}
determined by the pair $(f,x)$ of attaching maps, $f \in \pi_1(S^1) = \Z$ and $x \in \pi_2(P_f) = K \subseteq R$, as in (\ref{P_{f,x}}). Using results on totally free quadratic modules in \cite{Baues}, we investigate the structure of the third homotopy group  $\pi_3(P_{f,x})$ as a $\pi_1$--module by defining a set--theoretic splitting $u_x$ of J.H.C. Whitehead's Certain Exact Sequence of the universal cover, $\widehat{P}_{f,x}$, 
\begin{equation}\label{JHCses}
\xymatrix{
\Gamma(\pi_2(P_{f,x})) \,\, \ar@{>->}[r] & \pi_3(P_{f,x}) \ar@{->>}[r] & \text{H}_3(\widehat{P}_{f,x}) \ar@/^/[l]^{u_x}.}
\end{equation}

Recall that $\pi_1 = \pi_1(P_f) = \Z / f\Z$ with augmentation ideal $K = \ker f \varepsilon$, and let $B$ be the image of $d_x: R \rightarrow R, y \mapsto xy$. Then 
\begin{equation}\label{pi_2ofP_{f,x}}
\pi_2 (P_{f,x}) = \text{H}_2(\widehat{P}_{f,x}) = K / B = (\ker f \varepsilon) / xR.
\end{equation} 
The functor $\sigma$ in (IV 6.8) in \cite{Baues} assigns a totally free quadratic module $(\omega, \delta, \partial)$ to the pseudo--projective 3--space $P_{f,x}$ and we obtain the commutative diagram 
\[\xymatrix{
\Gamma(\pi_2 (P_{f,x})) \,\, \ar@{>->}[r] \ar@{>->}[d] & R \otimes R/\Delta_B  \ar@{->>}[r]^q \ar@{>->}[d] ^\omega & R \otimes R / \Gamma(K) \ar@{>->}[d]^w & \\
\pi_3(P_{f,x}) \,\, \ar@{>->}[r] \ar@{->>}[d] & L \ar[r]^\delta \ar@{->>}[d]  & M \ar[r]^\partial \ar@{->>}[d] & N \ar@{=}[d] \\
\text{H}_3(\widehat P_{f,x}) \,\, \ar@{>->}[r] \ar@/_/[u]_{u_x} & R \ar[r]^{d_x} \ar@/_/[u]_{t_x} & R \ar[r]^{f \cdot \varepsilon} \ar@/_/[u]_s & \Z
}\]
of straight arrows. Here the generators $e_3 \in L$, $e_2 \in M$ and $e_1 = 1 \in N = \Z$ correspond to the cells of $P_{f,x}$ and $\partial$ is the totally free nil(2)--module of Lemma \ref{ccmiso}. The right hand column is the short exact sequence (\ref{ses3}) with the set theoretic splitting $s$ defined in Section \ref{nil2comp}. The short exact sequence in the middle column is described in (IV 2.13) in \cite{Baues}, where the product $[\alpha, \beta]$ of $\alpha \in K$ and $\beta \in B$ is given by $[\alpha,\beta] = \alpha \otimes \beta + \beta \otimes \alpha \in R \otimes R$ and 
\[\Delta_B = \Gamma(B) + [K,B].\]  
By Corollary (IV 2.14)  in \cite{Baues}, taking kernels yields Whitehead's short exact sequence (\ref{JHCses}) in the left hand column of the diagram, that is, $\ker q = \Gamma(\pi_2(\widehat{P}_{f,x})), \ker \delta = \pi_3(P_{f,x})$ and $\ker d_x = \text{H}_3(\widehat{P}_{f,x})$. As $(\omega, \delta, \partial)$ is a quadratic module associated to $P_{f,x}$, we may assume that $\delta(e_3) = s(x)$.

In Section \ref{nil2comp} we determined the structure of $M$ as an $N$--module by computing the cross--effects of the set--theoretic splitting $s$ with respect to addition and the action. Analogously to the definition of $s$, we now define a set-theoretic splitting of the short exact sequence in the second column of this diagram by
\[t_x: R \longrightarrow L, \quad \sum_{k = 0}^{f-1} y_{\ol{k}} \, [\ol{k}] \longmapsto y_{\ol{0}} \, e_3^0  + \ldots + y_{\ol{f-1}} \, e_3^{f-1}.\]
The cross--effects of $t_x$ with respect to addition and the action determine the $N$--module structure of $L$, but we want to determine the module structure of $\pi_3(P_{f,x})$. To obtain a set-theoretic splitting of the first column which will allow us to do so, we must adjust $t_x$, such that the image of $\text{H}_3(\widehat P_{f,x})$ under the new splitting is contained in $\ker\delta = \pi_3(P_{f,x})$. Recall that $\delta$ is a homomorphism which is equivariant with respect to the action of $N$ and $\delta(e_3) = s(x)$. Thus Lemma \ref{mu} yields, for $y \in \text{H}_3(\widehat P_{f,x}) = \ker d_x$, that is, for $d_x(y) = xy = 0$,
\begin{eqnarray*}
\delta(t_x(y))
& = & \delta\big( \sum_{i = 0}^{f-1} y_i e^{\ol{i}}_3 \big) = \sum_{i = 0}^{f-1} y_i \delta(e_3)^{\ol{i}} = \sum_{i = 0}^{f-1} y_i \big(s(x)\big)^{\ol{i}} \\
& = & s(xy) + w(\mu(x,y)) \\
& = &  \delta \omega \mu(x,y).
\end{eqnarray*}
Hence $t_x(y) - \omega\mu(x,y) \in \ker\delta = \pi_3(P_{f,x})$, giving rise to the set theoretic splitting 
\[u_x: \text{H}_3(\widehat P_{f,x}) \longrightarrow \pi_3(P_{f,x}), \quad y \longmapsto t_x(y) - \omega\mu(x,y)\]
of the Hurewicz map $\pi_3 \twoheadrightarrow \text{H}_3$. The cross--effects of $u_x$ with respect to addition and the action determine (\ref{JHCses}) as a short exact sequence of $\pi_1$--modules. In Section \ref{quadcomp} we determine the cross--effects of $t_x$ and investigate the properties of the functions $A$ and $B$ describing the cross--effects of $u_x$. 

\section{Computations in Free Quadratic Modules}\label{quadcomp}
The first two results of this Section describe the cross--effects of $t_x$ with respect to addition and the action, respectively. We then turn to the properties of the cross--effects of $u_x$.
\begin{lem}\label{tcross1}
Take $z, y \in R$. Then, with the notation in (\ref{f(x|y)}), 
\[t_x(z|y) 
= \omega (\Psi(z,y)),\]
where
\[\Psi(z,y) = \sum_{m=1}^{f-1}\sum_{n=0}^{m-1} z_{\ol{m}} \, y_{\ol{n}} \, x [\ol{n}] \otimes x [\ol{m}].\]
Thus $\Psi(z,y)$ is linear in $z$ and $y$, yielding a homomorphism $\Psi: R \otimes R \rightarrow R \otimes R$.
\end{lem}
\begin{proof}
As in the proof of Lemma \ref{scross1}, we obtain
\[t_x(z|y) = \sum_{m=1}^{f-1}\sum_{n=0}^{m-1} z_{\ol{m}} \, y_{\ol{n}} (e_3^{\ol{n}}, e_3^{\ol{m}}).\]
Note that $\{\delta(e_3^{\ol{n}})\} = \{\delta(t_x([{\ol{n}}]))\} = d_x([{\ol{n}}]) = x[{\ol{n}}]$. Thus (\ref{qdef2}) yields
\begin{equation*}
t_x(z|y) 
= \sum_{m=1}^{f-1}\sum_{n=0}^{m-1} z_{\ol{m}} \, y_{\ol{n}}  \, \omega (\{\delta(e_3^{\ol{n}})\} \otimes \{\delta(e_3^{\ol{m}})\})
=  \sum_{m=1}^{f-1}\sum_{n=0}^{m-1} z_{\ol{m}} \, y_{\ol{n}}  \, \omega (x[\ol{n}] \otimes x[\ol{m}]).
\end{equation*}
\end{proof}

As $N = \Z$ is cyclic, the action of $N$ on $L$ is determined by the generator $1 \in \Z$. 
\begin{lem}\label{tcross2}
Take $x \in R$. Then
\[\big(t_x(y)\big)^1 = t_x(y^{\ol{1}}) + \omega (\ol{\Psi}_1(a,b)),\] 
where 
\[\ol{\Psi}_1 = \sum_{p=0}^{f-2} y_{\ol{p}} \, y_{\ol{f-1}} \, x [\ol{p+1}] \otimes x [\ol{0}]   
+ y_{\ol{f-1}} \,( x \otimes [\ol{0}] + [\ol{0}] \otimes x).\]
\end{lem}
\begin{proof}
With $\{\delta(e_3^{\ol{n}})\} = x[{\ol{n}}]$ from above and (\ref{qdef1}), we obtain
\begin{eqnarray*}
e_3^{1+f} 
& = & (e_3^1)^{f} 
= (e_3^1)^{\partial(e)} 
= e^1 + \omega(\{\delta(e_3^1)\} \otimes\{e\} + \{e\} \otimes \{\delta(e_3^1)\}) \\
& = & t_x([\ol{n}]^{\ol{1}})+ \omega( x[\ol{1}] \otimes [\ol{0}] + [\ol{0}] \otimes x[\ol{1}]).
\end{eqnarray*}
Thus, for $\ol{n} \in \pi$,
\begin{equation*}\label{tequa}
\big(t_x([\ol{n}])\big)^1 = 
\begin{cases}
\omega( t_x([\ol{n}]^{\ol{1}}) & \quad \text{for} \quad 0 \leq n < f-1, \\
\omega(t_x([\ol{n}]^{\ol{1}}) + x[\ol{1}] \otimes [\ol{0}] + [\ol{0}] \otimes x[\ol{1}]) & \quad \text{for} \quad n = f - \ell.
\end{cases}
\end{equation*}
With (\ref{qdef2}), we obtain, for $y = \sum_{n=0}^{f-1} y_{\ol{n}} [\ol{n}]$, 
\begin{eqnarray*}
\big(t_x(y)\big)^1 
& = & y_{\ol{0}} \, e_3^1 + y_{\ol{1}} \, e_3^2 \ldots + y_{\ol{f-2}} \, e_3^{f-1} + y_{\ol{f-1}} \, e_3^{f} \\
& = &  y_{\ol{0}} \, t_x([\ol{0}]^{\ol{1}}) + \ldots + y_{\ol{f-2}} \, t_x([\ol{f-1}]^{\ol{1}}) + y_{\ol{f-1}} \,t_x([\ol{f-1}]^{\ol{1}}) +  y_{\ol{f-1}} \, \omega( x \otimes [\ol{0}] + [\ol{0}] \otimes x) \\
& = & t_x(y^{\ol{1}})+ \sum_{p=0}^{f-2} y_{\ol{p}} \, y_{\ol{f-1}} \, (e_3^{p+1},e_3) + y_{\ol{f-1}} \omega( x \otimes [\ol{0}] + [\ol{0}] \otimes x) \\
& = & t_x(y^{\ol{1}}) + \sum_{p=0}^{f-2} y_{\ol{p}} \, y_{\ol{f-1}} \, x [\ol{p+1}] \otimes x [\ol{0}]   
+ y_{\ol{f-1}} \,( x \otimes [\ol{0}] + [\ol{0}] \otimes x)
\end{eqnarray*}
\end{proof}
The next two results concern the properties of the maps $A$ and $B$ which describe the cross--effects of $u_x$ with respect to addition and the action, respectively.
\begin{lem}
For $x \in K$ the map
\[A: \text{H}_3 \widehat P_{f,x} \times \text{H}_3 \widehat P_{f,x} \rightarrow \Gamma(\pi_2 P_{f,x}), (y,z) \mapsto u_x(y|z)\]
is bilinear. 
\end{lem}
\begin{proof}
Take $x \in K$ and $y, z \in \text{H}_3 \widehat P_{f,x}$. By definition 
\[A(y,z) = u_x(y|z) = t_x(y|z) - \omega \mu(x, y|z) = \omega \big( \Psi(y,z) - \mu(x, y|z)\big).\]
Thus Lemmata \ref{mu(x,y|z)} and \ref{tcross1} imply that A is bilinear. 
\end{proof}
\begin{lem}
For $x \in K$ define
\[B: \text{H}_3 \widehat P_{f,x} \rightarrow \Gamma(\pi_2 P_{f,x}), y \mapsto (u_x(y))^1 - u_x(y^1)\]
Then
\[ \text{H}_3 \widehat P_{f,x} \times \text{H}_3 \widehat P_{f,x} \rightarrow \Gamma(\pi_2 P_{f,x}), (y,z) \mapsto B(y|z)\]
is bilinear.
\end{lem}
\begin{proof}
Take $x \in K$ and $y, z \in \text{H}_3 \widehat P_{f,x}$ . Then
\begin{eqnarray*}
(A(y,z))^1 
& = & (u_x(y+z) - (u_x(y) + u_x(z))^1\\
& = & (u_x(y+z))^1 - (u_x(y))^1 - (u_x(z))^1 \\
& = & B(y+z) + u_x((y+z)^1) - (B(y) + u_x(y^1) + B(z) + u_z(z^1)).\\
& = & B(y|z) + A(y^1,z^1)
\end{eqnarray*}
Thus
\begin{eqnarray}
B(y|z) = (A(y,z))^1 - A(y^1,z^1)
\end{eqnarray}
and bilinearity follows from that of $A$ and the properties of an action.
\end{proof}

\section{Examples of Pseudo--Projective $3$--Spaces}\label{examples}
In this Section we provide explicit computations for examples of pseudo--projective $3$--spaces, including proofs for Theorem \ref{non-trivialpi1action}, Theorem \ref{AandBforex} and Theorem \ref{extcalc}.

Note that, as abelian group, the augmentation ideal $K$ of a pseudo--projective $3$--space $P_{f,x}$, as in (\ref{P_{f,x}}), is freely generated by $\{  [\ol{1}] - [\ol{0}], \ldots,  [\ol{f-1}] - [\ol{0}] \}$. We consider pseudo--projective $3$--spaces, $P_{f,x}$, with $x= \tilde{x}([\ol{1}]-[\ol{0}])$ and $\tilde x \in \mathbb Z$. We compute $\pi_2(P_{f,x}), \text{H}_3(\widehat P_{f,x})$, as well as the cross--effects of $u_x$ for this special case. For $f=2$, the general case coincides with the special case and provides an example where $\pi_1$ acts trivially on $\Gamma \pi_2 (P_{2, \tilde x})$ and on $\text{H}_3 (\widehat P_{2, \tilde x})$, but non--trivially on $\pi_3 (P_{2, \tilde x})$.

\begin{lem}\label{speciald_x}
For $x = \tilde{x}([\ol{1}]-[\ol{0}])$ with $\tilde x \in \mathbb Z$,
\[\text{H}_3(\widehat P_{f,x}) = \{ \tilde y N \, | \, \tilde y \in \mathbb Z\} \cong \Z,\]
is generated by the norm element $N = \sum_{k=0}^{f-1}[\ol{k}]$. Hence $\pi_1$ acts trivially on $\text{H}_3(\widehat P_{f,x})$. Furthermore,
\[\pi_2(P_{f,x}) = (\Z / \tilde x \Z) \otimes _\Z K.\]
Hence $\tilde x^2 \ell = 0$ for every $\ell \in \Gamma(\pi_2 (P_{f,x}))$.
\end{lem}
\begin{proof}
Take $x = \tilde{x}([\ol{1}]-[\ol{0}])$ with $\tilde x \in \mathbb Z$ and $y = \sum_{k=0}^{f-1} y_{\ol{k}} [\ol{k}] \in \ker d_x$. Then
\begin{eqnarray*}
d_x(y) = xy = 0
& \Longleftrightarrow & \tilde{x} \sum_{k=0}^{f-1} y_{\ol{k}} ([\ol{k} + \ol{1}] - [\ol{k}])  = 0\\
& \Longleftrightarrow & y_{\ol{f-1}} = y_{\ol{0}} = y_{\ol{1}} = y_{\ol{2}} = \ldots =  y_{\ol{f-2}} = \tilde y,\\
\end{eqnarray*}
for some $\tilde y \in \mathbb Z$. Hence $y = \tilde y N$.

By (\ref{pi_2ofP_{f,x}}), $\pi_2(P_{f,x}) = K/xR$. As abelian group, $K = \ker \varepsilon$ is freely generated by $\{ [\ol{k}] - [[\ol{0}]\}_{1 \leq k \leq f-1}$ and hence also by $\{ [\ol{k}] - [\ol{k-1}]\}_{1 \leq k \leq f-1}$. For $y = \sum_{i=0}^{f-1} y_{\ol{i}} [\ol{i}] \in R$ we obtain
\begin{eqnarray*}
xy 
& = &  \tilde x \sum_{i=1}^{f-1} y_{\ol{i}} ([\ol{i}] - [\ol{i-1}])  + \tilde x y_{\ol{f-1}}([\ol{0}] - [\ol{f-1}]) \\
& = &  \tilde x \sum_{i=1}^{f-1} y_{\ol{i}} ([\ol{i}] - [\ol{i-1}])  - \tilde x y_{\ol{f-1}} \sum_{i=1}^{f-1} ([\ol{i}] - [\ol{i-1}]) \\
& = & \tilde x  \sum_{i=1}^{f-1} (y_{\ol{i}} - y_{\ol{f-1}}) ([\ol{i}] - [\ol{i-1}]). \\
\end{eqnarray*}
As $\tilde x K \subseteq xR$, we obtain $xR = \tilde x K$ and hence
\[\pi_2(P_{f,x}) = K/xR = K / \tilde x K = (\Z / \tilde x \Z) \otimes _\Z K.\]
If $\tilde x$ is odd, then every element $\ell \in \Gamma(\pi_2 (P_{f,x}))$ has order $\tilde x$. If $\tilde x$ is even, an element $\ell \in \Gamma(\pi_2 (P_{f,x}))$ has order $2\tilde x$ or $\tilde x$. In either case, $\tilde x^2 \ell = 0$ for every $\ell \in \Gamma(\pi_2 (P_{f,x}))$.
\end{proof}
\begin{lem}\label{u_xcrossaddition}
Take $x= \tilde{x}([\ol{1}]-[\ol{0}])$ and $y, z \in  \text{H}_3(\widehat P_{f,x})$. Then 
\[A(y,z) = 0.\]
\end{lem}
\begin{proof}
By definition, 
\[A(y,z) = u_x(y|z) = t_x(y|z) - \omega \mu(x, y|z) = \omega ( \Psi(y,z) - \mu(x, y|z)).\] 
The definition of $\Psi$ and Lemma \ref{mu(x,y|z)} yield
\begin{equation*} 
\Psi(y,z) - \mu(x, y|z)) = \tilde y \tilde z \big(\sum_{p=1}^{f-1} \sum_{q=0}^{p-1} x [\ol{q}] \otimes x [\ol{p}] + 2  \sum_{q=1}^{f-1} \sum_{p=0}^{p-1} \nabla(x^{\ol{p}}, x^{\ol{q}}) - 2 \sum_{p=1}^{f-1} Q_p(x) + \sum_{p=0}^{f-1} (\nabla(x, x))^{\ol{p}} \big). 
\end{equation*}
Recall that $\tilde x^2 \ell = 0$ for every $\ell \in \Gamma(\pi_2 (P_{f,x}))$ and note that, by the properties of $Q$ and $\nabla$, each summand in the above sum has a factor of $\tilde x^2$. 
\end{proof}
\begin{lem}\label{u_xcrossaction}
Let $\gamma: \pi_2(P_{f,x}) \rightarrow \Gamma(\pi_2(P_{f,x}))$ be the universal quadratic map for the Whitehead functor $\Gamma$. Take $q: K \rightarrow \pi_2(P_{f,x}), k \mapsto 1 \otimes k, x = \tilde{x}([\ol{1}]-[\ol{0}])$ and $y = \tilde y N$. Then
\begin{equation*}
B(y) = - \tilde x \tilde y \gamma q ([\ol{1}] - [\ol{0}]).
\end{equation*}
\end{lem}
\begin{proof}
Note that $y^\beta = y$ for $\beta \in \pi_1$. As $\tilde x^2 \ell = 0$ for every $\ell \in \Gamma(\pi_2 (P_{f,x}))$, any summand with a factor $\tilde x^2$ is equal to $0$. By Lemma \ref{tcross2},
\begin{eqnarray*}
\ol{\Psi}_1(y) 
& = & \sum_{p=0}^{f-2} \tilde{y}^2 \big( \tilde{x}([\ol{1}]-[\ol{0}]) [\ol{p+1}] \otimes (\tilde{x}[\ol{1}]-[\ol{0}])  \big) + \tilde{y} \big(\tilde{x}([\ol{1}]-[\ol{0}]) \otimes [\ol{0}] + [\ol{0}] \otimes \tilde{x}([\ol{1}]-[\ol{0}]) \big) \\
& = & \tilde{x} \tilde{y} (([\ol{1}] - [\ol{0}]) \otimes [\ol{0}] + [\ol{0}] \otimes ([\ol{1}] - [\ol{0}])).
\end{eqnarray*}
Lemma \ref{mu} yields
\begin{eqnarray*}
\mu(x, y) 
& = & - \sum_{q=0}^{f-1} \sum_{p=0}^{q-1} \tilde x^2 \tilde{y}^2 \nabla \big(([\ol{p+1}] - [\ol{p}]), ([\ol{q+1}] - [\ol{q}]) \big) + \sum_{p=0}^{f-1}  \ol{\nabla}_p \big( \tilde{y}\tilde{x}([\ol{1}] - [\ol{0}]) \big) \\
&& \quad \quad - \tilde x^2 \binom{\tilde{y}}{2} \big(\nabla (([\ol{1}] - [\ol{0}]), ([\ol{1}] - [\ol{0}]))\big)^{\ol{p}} \\
& = &  \ol{\nabla}_{f-1} \big(\tilde{x} \tilde{y}([\ol{1}] - [\ol{0}])\big) \\
& = & - \, \tilde{x}^2 \tilde{y}^2 \big( [\ol{f-1}] \otimes [\ol{0}] - [\ol{0}] \otimes [\ol{0}] \big) + \tilde{x} \tilde{y}  \,[\ol{0}] \otimes [\ol{0}]  \\
& = & \tilde{x} \tilde{y}  \,[\ol{0}] \otimes [\ol{0}]. \\
\end{eqnarray*}
Thus
\[B(y) = (u_x(y))^1 - u_x(y^{\ol{1}}) = \omega\big( \ol{\Psi}_1(y) - (\mu(x,y))^{1} + \mu(x,y)\big) = - \tilde x \tilde y \gamma q ([\ol{1}] - [\ol{0}]).\]
\end{proof}

Together Lemmata \ref{speciald_x}, \ref{u_xcrossaddition} and \ref{u_xcrossaction} provide a proof of Theorem \ref{AandBforex}.

For $f=2$ the special case coincides with the general case and we obtain 
\begin{thm}\label{f=2}
Let $X = P_{2,x}$ be a pseudo--projective $3$--space with $x = \tilde x ([\ol{1}] - [\ol{0}])$, for $\tilde x \in \Z$ and $\tilde x \neq 0$. Then $u_x$ is a homomorphism and the fundamental group $\pi_1 = \Z / 2 \Z$ acts trivially on $\Gamma (\pi_2 P_{2, x})$ and on $\text{H}_3 \widehat P_{2, x}$. The action of $\pi_1$ on $\pi_3 P_{2, x}$ is non--trivial  if and only if $\tilde x$ is even.
\end{thm}
\begin{proof}
For $f=2$ the augmentation ideal $K$ is generated by $k = [\ol{1}]-[\ol{0}]$. Since $k[\ol{1}] = -k$, the action of $\pi_1 = \Z / 2 \Z$ on $K$ and hence on $\pi_2 P_{2, x} = K/xR = \Z / \tilde x \Z$ is multiplication by $-1$. As the $\Gamma$--functor maps multiplication by $-1$ to the identity morphism, the action on $\pi_1$ on $\Gamma (\pi_2 P_{2, x})$ is trivial. The group $\text{H}_3 \widehat P_{2, x}$ is generated by the norm element $N = [\ol{0} ]+ [\ol{1}]$. As $N [\ol{1}] = N$, $\pi_1$ acts trivially on $\text{H}_3 \widehat P_{2, x}$. As $\pi_2 = \Z / \tilde x \Z$ is cyclic, $\Gamma \pi_2 = \pi_2$ if $\tilde x$ is odd and $\Gamma \pi_2 = \Z / 2 \tilde x \Z$ if $\tilde x$ is even, that is,
\begin{equation}\label{Gammapi2}
\Gamma \pi_2 = \Z / \gcd (\tilde x, 2) \tilde x \, \Z. 
\end{equation}
By Lemma \ref{u_xcrossaction} and (\ref{Gammapi2}), the action of $\pi_1$ on $\pi_3 X$ is non--trivial if and only if $\tilde x$ is even.
\end{proof}

Theorem \ref{non-trivialpi1action} is a corollary to Theorem \ref{f=2}.

\begin{proof}[Proof of \ref{extcalc}]
Note that $\Z / \tilde x \Z \otimes_\Z K$ is generated by $\{\alpha_k = q([\ol{k}] - [\ol{k-1}])\}_{0<k<f}$, where $q: K \rightarrow \Z /\tilde x \Z \otimes_\Z K, k \mapsto 1 \otimes k$. Thus
$\Gamma(\pi_2(P_{f,x})) = \Gamma(\Z / \tilde x \Z \otimes K) \subseteq (\Z /\tilde x \Z \otimes_\Z K) \otimes (\Z /\tilde x \Z \otimes_\Z K)$ is generated by $\{ \gamma q(\alpha_k), [q(\alpha_j), q(\alpha_k)]\}_{0<j<k, 0<k<f}$. With $\alpha_k^1 = \alpha_{k+1}$ for $1<k<f-1$ and $\alpha_{f-1}^1 = [\ol{0}] - [\ol{f-1}] = -\sum_{i=1}^{f-1} \alpha_i$, we obtain, for $\ell = \sum_{k=1}^{f-1} \ell_k \gamma(\alpha_k) + \sum_{k=2}^{f-1} \sum_{j=1}^{k-1} \ell_{j,k} [\alpha_j, \alpha_k] \in \Gamma(\pi_2(P_{f,\tilde x}))$,
\begin{eqnarray*}
\ell^1 - \ell
& = & \sum_{k=1}^{f-1} \ell_k \gamma q(\alpha_k)^1 + \sum_{k=2}^{f-1} \sum_{j=1}^{k-1} \ell_{j,k} \,\,[q(\alpha_j), q(\alpha_k)]^1 - \sum_{k=1}^{f-1} \ell_k \gamma q(\alpha_k) - \sum_{k=2}^{f-1} \sum_{j=1}^{k-1} \ell_{j,k} [q(\alpha_j), q(\alpha_k)]\\
& = & \sum_{k=1}^{f-2} \ell_k \gamma q(\alpha_{k+1}) + \ell_{f-1} \gamma q(- \sum_{i=1}^{f-1} \alpha_i) + \sum_{k=2}^{f-2} \sum_{j=1}^{k-1} \ell_{j,k} \,\,[q(\alpha_{j+1}), q(\alpha_{k+1})] \\
&& + \sum_{j=1}^{f-1} \ell_{j,f-1} [\gamma q(\alpha_{j+1}), \gamma q(-\sum_{i=1}^{f-1} \alpha_i)] - \sum_{k=1}^{f-1} \ell_k \gamma q(\alpha_k) - \sum_{k=2}^{f-1} \sum_{j=1}^{k-1} \ell_{j,k} [q(\alpha_j), q(\alpha_k)]\\
& = & (\ell_{f-1} - \ell_1) \gamma q(\alpha_1) + \sum_{k=2}^{f-1} (\ell_{k-1} - \ell_k + \ell_{f-1} - 2 \ell_{k-1,f-1}) \gamma q(\alpha_k) \\
&& + \sum_{k=2}^{f-1} (\ell_{f-1} - \ell_{1,k} - \ell_{k-1,f-1}) [q(\alpha_1, q(\alpha_k)] \\
&& + \sum_{k=3}^{f-1} \sum_{j=2}^{k-1} (\ell_{f-1} + \ell_{j-1,k-1} - \ell_{j,k} - \ell_{j-1,f-1} -\ell_{k-1,f-1}) [q(\alpha_j), q(\alpha_k)]. 
\end{eqnarray*}
Thus the sequence (\ref{ses1}) splits if and only if there is at least one solution of the system of equations
\[\begin{array}{lll}
(A) & 0 =  \ell_{f-1} - \ell_1 & \mod 2 \tilde x  \\
(B_k) & 0 =  \ell_{k-1} - \ell_k + \ell_{f-1} - 2 \ell_{k-1,f-1} & \mod 2 \tilde x \,\ \text{for} \, 2 \leq k \leq f-1\\
(C_k) & 0 =  \ell_{f-1} - \ell_{1,k} - \ell_{k-1,f-1} & \mod \tilde x  \,\, \text{for} \, 2 \leq k \leq f-1\\
(D_{j,k}) & 0 =  \ell_{f-1} + \ell_{j-1,k-1} - \ell_{j,k} - \ell_{j-1,f-1} -\ell_{k-1,f-1} & \mod \tilde x  \,\,\text{for} \, 2 \leq j \leq k, 2 < k < f-1.\\
\end{array}\]
For odd $f$, a solution of the system is given by $\ell_{j,k} = 0$ for $1 \leq j \leq k-1, 1<k<f-1, \ell_k = 0$ for $k$ odd, and $\ell_k = \tilde x$ for $k$ even. Hence (\ref{ses1}) splits if $f$ is odd. It remains to show that there are no solutions for even $f>2$.

For $2 \leq j < \frac{1}{2}(f-2)$, subtract the equation $(D_{i,f-j+i})$ from the equation $(D_{i,f-j+i-1})$ for $2 \leq i < j$. Add $(D_{j,f-1})$ and $(C_{f-j})$, then subtract $(C_{f-j+1})$. Adding the resulting equations yields
\[(E_{j})  \quad 0 = \ell_{f-1} - \ell_{j,f-1} - \ell_{f-j-1,f-1} \mod \tilde x.\]
Multiplying the equations $(C_{f-1})$ and $(E_j), 2 \leq j \leq \frac{1}{2}(f-2)$ by $2$ and adding them we obtain
\[0 = (f-2)\ell_{f-1} - 2 \sum_{j=1}^{f-2} \ell_{j,f-1} \mod 2 \tilde x.\]
On the other hand, adding the equations $(A)$ and $(B_k), 1<k<f-1$, the resulting equation is
\[\tilde x = (f-2)\ell_{f-1} - 2 \sum_{j=1}^{f-2} \ell_{j,f-1} \mod 2 \tilde x.\]
Hence there are no solutions for $f$ even.
\end{proof}

\section{Pseudo--Projective Spaces in Dimension $4$}
In the final section we consider $4$--dimensional pseudo--projective spaces and provide a proof of Theorem \ref{trivialonpi_3}. We begin by constructing a $4$--dimensional pseudo--projective space associated to given algebraic data. Namely, take $f \in \Z$ with $f \geq 0, x, y \in R = \Z[\Z/f\Z]$ with $xy = 0$ and $f\varepsilon(x) = 0$, where $\varepsilon$ is the augmentation of the group ring, $R$, so that $xR \subseteq \ker \varepsilon$. Finally, take $\gamma \in \Gamma((\ker f \varepsilon) / xR)$. Given such data, $(f, x, y, \alpha)$, take a $3$--dimensional pseudo--projective space $P_{f,x}$ as in (\ref{P_{f,x}}). Then the set--theoretic splitting $u_x$ of the short exact sequence
\begin{equation*}
\xymatrix{
\Gamma(\pi_2 ({P}_{f,x})) \,\, \ar@{>->}[r] & \pi_3 ({P}_{f,x}) \ar@{->>}[r] & \text{H}_3(\widehat{P}_{f,x})}
\end{equation*}
implies that every element of $\pi_3 ({P}_{f,x})$ may be expressed uniquely as a sum $u_x(v) + \beta$ with $v \in  \text{H}_3(\widehat{P}_{f,x})$, that is, $xv = 0$, and $\beta \in \Gamma(\pi_2 ({P}_{f,x})) = \Gamma((\ker f \varepsilon) / xR)$, see (\ref{pi_2ofP_{f,x}}). Using $u_x(y) + \alpha \in \pi_3 ({P}_{f,x})$ to attach a 4--cell to $P_{f,x}$ we obtain the $4$--dimensional pseudo--projective space,
\[P = P_{f,x,y,\alpha} = S_1 \cup e^2 \cup e^3 \cup e^4.\]
Note that the homotopy type of $P = P_{f,x,y,\alpha}$ is determined by $(f, x, y, \alpha)$ and that every 4--dimensional pseudo--projective space is of this form. The cellular chain complex, $C_{\ast}(\widehat P)$, of the universal cover, $\widehat{P} = \widehat{P}_{f,x,y,\alpha}$, is the complex of free $R$--modules, 
\begin{equation*}
\xymatrix{
\langle e_4 \rangle_R \ar[r]^{d_4} & \langle e_3 \rangle_R \ar[r]^{d_3} & \langle e_2 \rangle_R \ar[r]^{d_2} & \langle e_1 \rangle_R \ar[r]^{d_1} & \langle e_0 \rangle_R,}
\end{equation*}
given by $d_1(e_1) = e_0([\ol{1}] - [\ol{0}]), d_2(e_2) = e_1 N$, that is, multiplication by the \emph{norm element}, $N = \sum_{i=0}^{f-1} [\ol{i}], d_3(e_3) = e_2 x$, and $d_4(e_4) =  e_3 y$. Let $\bar b: R \rightarrow \pi_3 P_{f,x}$ be the homomorphism of $R$--modules which maps the generator $[\ol{0}] \in R$ to $\bar b([\ol{0}]) = u_x(y) + \alpha$, so that composition with the projection onto $ \text{H}_3 \widehat P_{f,x}$ yields the homomorphism of $R$--modules induced by the boundary operator $d_4$. Thus we obtain the commutative diagram 
\begin{equation*}
\xymatrix{
\text{H}_4\widehat P \ar[r]^b \ar@{>->}[d] & \Gamma \pi_2 P \ar[dr]^j  \ar@{>->}[d]  & \\
R \ar[r]^{\bar{b}} \ar[dr]_{\bar{d_4}} & \pi_3 P_{f,x} \ar@{->>}[r] \ar@{->>}[d] & \pi_3 P \ar@{->>}[d]^h \\
& \text{H}_3 \widehat P_{f,x}  \ar@{->>}[r] & \text{H}_3 \widehat P  }
\end{equation*}
in the category of $R$--modules, where the middle column is the short exact sequence (\ref{JHCses}) and 
\begin{equation}\label{JHCses2}
\xymatrix{
\text{H}_4\widehat{P} \ar[r]^b & \Gamma\pi_2 P \ar[r]^j & \pi_3 P \ar@{->>}[r]^h & \text{H}_3 \widehat{P}}
\end{equation}
is Whitehead's Certain Exact Sequence of the universal cover, $\widehat{P} = \widehat{P}_{f,x,y,\alpha}$.

Now we restrict attention to the case $f = 2$. Then $\pi_1 = \pi_1 P = \Z / 2\Z$ and the augmentation ideal, $K$ is generated by $[\ol{1}] - [\ol{0}]$. Thus
\begin{equation*}
x = \tilde x([\ol{1}] - [\ol{0}]) \quad \t{and} \quad y = \tilde y([\ol{1}] + [\ol{0}]), \quad \text{for some} \,\, \tilde x, \tilde y \in \Z.
\end{equation*}
We assume that $x$ and $y$ are non--trivial, that is, $\tilde x, \tilde y \neq 0$.
\begin{thm}\label{trivialonpi_3detail}
For $P = P_{2,x,y,\alpha}$, with $x$ and $y$ as above, $\pi_1 P = \Z / 2 \Z$ acts on $\pi_2 P = \Z / \tilde x Z$ via multiplication by $-1$, trivially on $\t{H}_3 \widehat P = \Z / \tilde y \Z$ and via multiplication by $-1$ on $\t{H}_4 \widehat P = \Z = \langle [\ol{1}] - [\ol{0}] \rangle$. The exact sequence (\ref{JHCses2}) is given by
\begin{equation}\label{JHCses3}
\xymatrix{
\text{H}_4\widehat{P} = \Z \ar[r]^-b & \Gamma\pi_2 P = \Gamma(\Z/ \tilde x \Z) \ar[r]^-j & \pi_3 P \ar@{->>}[r]^-h & \text{H}_3 \widehat{P} = \Z / \tilde y \Z.}
\end{equation}
Denoting the generator of $\Gamma\pi_2 P$ by $\xi$, the boundary $b$ is determined by 
\[b([\ol{1}] - [\ol{0}]) = \tilde x \tilde y \xi,\]
and the action of $\pi_1 P$ on $\pi_3 P$ is trivial. 
As abelian group, $\pi_3 P$ is the extension of $\t{H}_3 \widehat P$ by $\t{coker} \, b$ given by the image of $-\alpha \in \Gamma \pi_2$ under the homomorphism
\begin{equation*}
\xymatrix{
\tau: \Gamma \pi_2 \ar@{->>}[r] & \t{coker}\,b \ar@{->>}[r] & \t{coker}\,b / \tilde y \t{coker}\,b = \t{Ext}(\Z / \tilde y \Z, {\rm{coker}} \, b).}
\end{equation*}
Hence the extension $\pi_3 P$ over $\Z$ determines $\alpha$ modulo $\ker \tau$.
\end{thm}
Theorem \ref{trivialonpi_3} is a corollary to Theorem \ref{trivialonpi_3detail}.
\begin{proof}
As the augmentation ideal $K \cong \Z$ is generated by $k = [\ol{1}] - [\ol{0}]$, the action of $\pi_1 = \Z / 2 \Z$ on $K = \pi_2 P_2$ and hence on $\pi_2 P = K / x R = \mathbb Z / \tilde x \mathbb Z$ is multiplication by $-1$, since $k [\ol{1}] = -k$. But the $\Gamma$--functor maps mutliplication by $-1$ to the identity morphism, so that $\pi_1$ acts trivially on $\Gamma (\pi_2 P)$. 

As $d_3(e_3) = e_2 x$, we obtain $\text{H}_3 \widehat P_{2,x} \cong \Z$, generated by the norm element $N = [\ol{1}] + [\ol{0}]$. Since $N [\ol{1}] = N$, the action of $\pi_1$ on $\text{H}_3 \widehat P_{2,x}$ is trivial. 

As $d_4(e_4) = e_3 y$, we obtain $\text{H}_3 \widehat P \cong \Z / \tilde y \mathbb Z$ and $\text{H}_4\widehat{P} \cong \Z$, generated by $k = [\ol{1}] - [\ol{0}]$. Hence the action of $\pi_1$ on $\text{H}_4\widehat{P}$ is multiplication by $-1$.

Now let $\xi = ([\ol{1}] - [\ol{0}]) \otimes ([\ol{1}] - [\ol{0}])$ be the generator of $\Gamma(K)$. Note that $v[\ol{1}] = v$ and $\beta [\ol{1}] = \beta$, for $v \in \text{H}_3 \widehat P_{2,x}$ and $\beta \in \Gamma (\pi_2 P)$, since $\pi_1$ acts trivially on both $\text{H}_3 \widehat P_{2,x}$ and $\Gamma (\pi_2 P)$. Lemma \ref{u_xcrossaction} implies
\[(u(v) + \beta) [\ol{1}] = - \tilde x \tilde y \, \omega(\xi) + u(v[\ol{1}]) + \omega(\beta) [\ol{1}] = - \tilde x \tilde y \, \omega(\xi) + u(v) + \omega(\beta).\]
We obtain
\begin{eqnarray*}
\bar b (e_4 ([\ol{1}] - [\ol{0}]))
& = & (u(y) + \omega(\alpha)) ([\ol{1}] - [\ol{0}]) \\
& = & - \tilde x \tilde y \, \omega(\xi) + u(y) + \omega(\alpha) - (u(y) + \omega(\alpha)) \\
& = & - \tilde x \tilde y \, \omega(\xi).
\end{eqnarray*}
By definition of $\bar{b}$,
\[\pi_3 P = \pi_3 P_{2,x} / {\rm im} \, \bar b.\]
Hence $\pi_1$ acts trivially on $\pi_3(P)$.

Sequence (\ref{JHCses2}) yields the short exact sequence
\begin{equation}\label{gammaext}
\xymatrix{
G = {\rm{coker}} \, b \,\,  \ar@{>->}[r] & \pi_3 P \ar@{->>}[r]^-h & \text{H}_3 \widehat P \cong \Z / \tilde y \Z,}
\end{equation}
which represents $\pi_3 P$ as an extension of $\Z / \tilde y \Z$ by $G = {\rm{coker}} \, b$. Thus the extension $\pi_3 P$ over $\Z$ determines $\gamma$ modulo the kernel of the map
\[\xymatrix{\tau: \Gamma \pi_2 \ar@{->>}[r] & \t{coker}\,b \ar@{->>}[r] & \t{coker}\,b / \tilde y \t{coker}\,b = \t{Ext}(\Z / \tilde y \Z, {\rm{coker}} \, b)}.\]
\end{proof}

\end{document}